\documentclass[]{amsart}

\def\br{\mathbb{R}} 
\def\bc{\mathbb{C}} 
\def\h{\mathcal{H}} 
\def\k{\mathcal{K}}

\newcommand{\inner}[2]{\langle #1,#2\rangle} 

\newcommand{\bh}{{\rm B}(\mathcal{H})}
\newcommand{\bk}{{\rm B}(\mathcal{K})}
\newcommand{\tr}{{Tr\,}}

\newtheorem{theorem}{Theorem}[section]
\newtheorem{lemma}[theorem]{Lemma}
\newtheorem{co}[theorem]{Corollary}
\newtheorem{pr}[theorem]{Proposition}
\theoremstyle{remark}
\newtheorem{re}[theorem]{Remark}
\theoremstyle{definition}

\numberwithin{equation}{section}

\begin{document}

\title[]{Sums of products of positive operators and spectra of L\"uders operators} 
\author{Bojan Magajna} 
\address{Department of Mathematics\\ University of Ljubljana\\
Jadranska 21\\ Ljubljana 1000\\ Slovenia}
\email{Bojan.Magajna@fmf.uni-lj.si}

\thanks{Acknowledgment. The author is very grateful to professor Heydar Radjavi for
the discussion concerning sums of projections and positive operators. He is also grateful
to Roman Drnov\v sek for bringing the reference \cite{W} to his attention.}




\begin{abstract}Each bounded operator $T$ on an infinite dimensional Hilbert space 
$\h$  is a sum of three operators that are
similar to positive operators; two such operators are sufficient if $T$ is not a compact
perturbation of a scalar. The spectra of L\"uders operators 
(elementary operators on $\bh$ with positive coefficients) of
lengths at least three are not necessarily contained in $\br^+$. On the other hand, the
spectra of such operators of lengths (at most) two are contained in $\br^+$ if
the coefficients on one side commute.
\end{abstract}

\maketitle

\section{Introduction}

Completely positive maps on $\bh$ (the algebra of all bounded operators
on a  Hilbert space $\h$) of the form
\begin{equation}\label{01}\Psi(X)=\sum_{j=1}^{n} A_j^*XA_j,\end{equation}
have received a renewed  interest recently especially in connection with quantum information theory 
(see  \cite{K},  \cite{LZ}, \cite{Pr2},
\cite{WJ}  and the references there). If all the coefficients 
$A_j$
in (\ref{01}) are positive operators
such a map is called a L\"uders operation. If $n$ is finite then these are special cases of 
elementary operators, that is,
maps of the form $X\mapsto\sum_{j=1}^n A_jXB_j$, whose spectra have been
intensively studied in the past (see \cite{Cur} and the references there), 
but only in the cases when both families of coefficients $(A_j)$ and $(B_j)$ are 
commutative. If $\h$ is finite dimensional, then $\bh$ is a Hilbert space for the inner
product induced by the trace  and it is easily verified that an elementary operator with
positive coefficients $A_j$ and $B_j$ is a positive operator on this Hilbert space, so
its spectrum is contained in $\br^+:=[0,\infty)$. 

At the end of
the paper \cite{N} it was asked if the spectrum of a L\"uders operator 
$X\mapsto\sum_{j=1}^n A_jXA_j$ with positive
coefficients on $\bh$ is necessarily contained in $\br^+$ if $\h$ is infinite dimensional. 
We will show that, contrary to what one might expect, the answer to this question
is negative.  This will be a consequence of the fact that the operator $T=-1$ can be expressed as
\begin{equation}\label{02}T=\sum_{j=1}^n A_jB_j\ \ \mbox{with positive}\ A_j, B_j\in\bh.\end{equation}  
At first the author did not
know how do prove
that every operator $T\in\bh$ is of the form (\ref{02}), but then professor
Heydar Radjavi told him that by \cite{S} and \cite{PT} $T$ is a 
sum of finitely many idempotents  and, since every idempotent is similar to a projection,
$T$ is a sum of products of 
positive operators. To see this, note that an operator $Q$ which is similar to a positive 
operator, say $Q=SPS^{-1}$, is a product of two positive operators: 
$Q=(SS^*)((S^{-1})^*PS^{-1})$.
By Pearcy and Topping \cite{PT}) five idempotents 
are sufficient to express any $T$ in this way and according to 
\cite[Proposition 5.9]{W} this is the minimal number since scalars are in general not 
sums of less than five idempotents. However, since 
idempotents are very special elements, we can not expect that $5$ is the minimal $n$
in (\ref{02}). 

One of the goals of this paper is to find the minimal $n$ above. The result will
imply that even the spectrum of a L\"uders operator of 
small length is not necessarily contained in $\br^+$. 
More precisely, in the next section we will show that every $T\in\bh$ is a sum of three operators $T_j$ 
each of which is similar to a positive operator. Moreover, if $T$ is not a compact perturbation
of a scalar, two operators $T_j$ are sufficient.  This result is optimal
since compact perturbations of nonzero scalars can not be expressed in the form (\ref{02}) with
$n\leq2$. We will also show that the trace class operators with trace not in $\br^+$ can not be 
expressed
as $T_1+T_2$ with both $T_1$ and $T_2$ similar to positive operators in $\bh$. 
As a preliminary step in the proof of the main result we will first show that $T$ is
a sum of four operators $T_j$ similar to positive ones, with some additional properties needed.

In the last section we will first apply this result to answer the above mentioned question
from \cite{N}. Then we will prove that
the spectra of operators of the form $X\mapsto \sum_{j=1}^2A_jXB_j$ with positive $A_j$ and
$B_j$ are contained in 
$\br^+$ if $A_1A_2=A_2A_1$ (or if $B_1B_2=B_2B_1$). 

Throughout the paper $\h$ denotes an infinite dimensional separable
Hilbert space and $\bh$ the algebra of all bounded linear operators on $\h$. (The results  hold
also for non separable $\h$, but in their formulations
the ideal of compact operators must be replaced by the unique proper maximal ideal of $\bh$.)
An operator $T\in\bh$ is called positive if $\inner{T\xi}{\xi}\geq0$ for all
$\xi\in\h$ (thus $T$ is not necessarily definite) and the set of all positive operators
is denoted by $\bh^+$.

\section{Sums of operators similar to positive operators}

We begin with a simple and well-known observation. Let $S\in{\rm B}
(\k\oplus\k)$ be a $2\times 2$ operator matrix
\begin{equation}\label{1}S=\left[\begin{array}{ll}
u&x\\
y&z\end{array}\right],\end{equation}
where $u$ is invertible. Then $S$ is invertible if and only if $z-yu^{-1}x$ is invertible
and in this case
\begin{equation}\label{2}S^{-1}=\left[\begin{array}{cc}
u^{-1}(1+xdyu^{-1})&-u^{-1}xd\\
-dyu^{-1}&d\end{array}\right],\ \ \ \mbox{where}\ d=(z-yu^{-1}x)^{-1}.\end{equation}
To prove this, multiply $S$ from the left by the invertible matrix
$$\left[\begin{array}{cc}
u^{-1}&0\\
-yu^{-1}&1\end{array}\right]$$
to obtain an upper-triangular matrix with $1$ and $z-yu^{-1}x$ along the diagonal.

The main assertion of the following lemma can be deduced from the proof of Theorem 1 in \cite{PT},
but later we will need some additional information  from its proof in
the form presented below.

\begin{lemma}\label{le1}Every operator $T\in\bh$ is a sum of the form $$T=\sum_{j=1}^4
S_jT_jS_j^{-1},$$ where $S_j\in\bh$ and the operators $T_j\in\bh$ are positive with
disjoint spectra $\sigma(T_j)$, each $\sigma(T_j)$ consists of at most two points,
$\sigma(T_1)\subset[0,1]$ and $\sigma(T_j)\subset(1,\infty)$ for $j\ne1$. Moreover, the
range of $T_1$ is closed and has infinite dimension and codimension.

In particular, $T$ can be written as $T=\sum_{j=1}^4A_jB_j$, where $A_j,B_j\in\bh^+$.
\end{lemma}

\begin{proof}Decompose $\h$ into an orthogonal sum of two isomorphic closed
subspaces, $\h=\k\oplus\k$; then $T$ is represented by an operator matrix of the form
\begin{equation}\label{3}T=\left[\begin{array}{cc}
A&B\\
C&D
\end{array}\right].
\end{equation}
It suffices to find diagonal positive operators $T_j=a_j\oplus b_j$ ($a_j,b_j\in\bk$)
and invertible operators $S_j$ ($j=1,2,3,4$) of the form (\ref{1}) such that 
$T=\sum_{j=1}^4S_jT_jS_j^{-1}.$ It turns out that we can even take $S_j$ of the form
$$S_j=\left[\begin{array}{cc}
1&x_j\\
y_j&1+y_jx_j
\end{array}\right].$$
Then 
$$S_jT_jS_j^{-1}=\left[\begin{array}{cc}
a_j+s_jy_j&-s_j\\
y_ja_j-b_jy_j+y_js_jy_j&b_j-y_js_j\end{array}\right],\ \ \mbox{where}\ s_j:=a_jx_j-x_jb_j.$$
There are many appropriate choices for $x_j,y_j,z_j,a_j,b_j$ in order to make
the sum $\sum_{j=1}^4S_jT_jS_j^{-1}$ equal to $T$. For example, if we let
$y_1=0=x_2$, $y_3=1$, $b_1=0$ and for $j\geq2$ choose all $a_j$ and $b_j$ to be positive scalars
with $a_j-b_j=1$, and denote $\beta=\sum_{j=2}^4b_j$ (so that $\sum_{j=2}^4a_j=\beta+3$),
then
\begin{equation}\label{4}\sum_{j=1}^4S_jT_jS_j^{-1}=\left[\begin{array}{cc}
a_1+\beta+3+x_3+x_4y_4&-a_1x_1-x_3-x_4\\
y_2+x_3+1+y_4x_4y_4&\beta-x_3-y_4x_4\end{array}\right].
\end{equation}
To achieve that the matrix in (\ref{4}) will be equal to $T$, we only need to choose
$x_3, x_4, y_4$  in $\bk$ and  invertible $a_1\in\bk^+$ so that 
\begin{equation}\label{5}a_1+\beta+3+x_3+x_4y_4=A\ \ \mbox{and}\ \ \beta-x_3-y_4x_4=D,
\end{equation}
for then the off-diagonal terms of the matrix (\ref{4}) can be made equal to $B$ and $C$
by a suitable choice of $y_2$ and $x_1$. Adding the two equations (\ref{5}) we see,
that we only need to choose $x_4,y_4$ and $a_1$ so that 
\begin{equation}\label{6}x_4y_4-y_4x_4=A+D-a_1-2\beta-3=:T_0,
\end{equation} for then $x_3$ can be computed from either of the equations (\ref{5}).
So (for a fixed $\beta$), we first choose an invertible positive $a_1\in\bk$ of the form
$\lambda+\mu p$, where $\lambda,\mu\in\br^+$ and $p$ is a projection of infinite rank
and nullity, such that $\sigma(a_1)\subset(0,1]$ and $T_0$ is not a compact
perturbation of a scalar. Then $T_0$ is a commutator by \cite{BP} (a simplified proof 
is in \cite{AS}), which means that there exist $x_4$ and $y_4$ satisfying (\ref{6}).
By suitably choosing scalars $a_j$ and $b_j$ ($j\geq2$) we can make the spectra of $T_j$ 
disjoint for all $j$.  
\end{proof}

\begin{re}\label{re1}For a later use observe that in the above proof the spectra of $a_j$ and $b_j$ 
are disjoint for all $j$; in fact all $a_j$ and $b_j$
chosen above are scalars, except possibly $a_1$. Also note that the operator
$S_1T_1S_1^{-1}$ has the form
$$\left[\begin{array}{cc}
a_1&*\\
0&0\end{array}\right],$$
where $a_1\in\bk^{+}$. 
\end{re}

\begin{theorem}\label{th2}Every $T\in\bh$ is of the form $T=\sum_{j=1}^3S_jT_jS_j^{-1},$
where $S_j\in\bh$ and the operators $T_j\in\bh$ are positive (and invertible for $j\leq2$) 
with finite spectra $\sigma(T_j)$, each $\sigma(T_j)$ consists
of at most four points. Moreover, $0$ is an isolated point of $\sigma(T_3)$, the range of $T_3$ 
is closed and has infinite dimension and codimension.
\end{theorem}

\begin{proof}As in the proof of Lemma \ref{le1} we represent $T$ by the operator matrix
(\ref{3}). Now we try to find positive block-diagonal operators $T_j=a_j\oplus b_j$ and
invertible operators $S_j\in\bh$ of the form (\ref{1}) (with $z-yu^{-1}x=1$) such that $\sum_{j=1}^3S_jT_jS_j^{-1}=T$.
Denoting 
$$S_j=\left[\begin{array}{cc}
u_j&x_j\\
y_j&z_j\end{array}\right],\ \mbox{where}\ u_j\ \mbox{is invertible}\ \mbox{and}\ 
z_j-y_ju_j^{-1}x_j=1,$$
we compute (using (\ref{2})) that
$$S_jT_jS_j^{-1}=\left[\begin{array}{cc}
c_j+s_jv_j&-s_j\\
v_jc_j-b_jv_j+v_js_jv_j&b_j-v_js_j\end{array}\right],$$
where
\begin{equation}\label{7}c_j:=u_ja_ju_j^{-1},\ \ v_j:=y_ju_j^{-1},\ 
\ \mbox{and}\ \ s_j:=c_jx_j-x_jb_j.
\end{equation}
Note that if the spectra of $b_j$ and $c_j$ are disjoint, then from 
(\ref{7}) $a_j$, $y_j$, $b_j$ and $x_j$ can all be computed from 
$c_j$, $u_j$, $v_j$, $b_j$, and $s_j$. (That the equation $c_jx_j-x_jb_j=s_j$
can be solved for $x_j$ is Rosenblum's theorem \cite[p. 8]{RR}.) 
Further, we assume that the matrix $S_3$ is diagonal (that is, $x_3=0=y_3$, so we will
only need that the spectra of $c_j$ and $b_j$ are disjoint for $j=1,2$). Then
the condition $\sum S_jT_jS_j^{-1}=T$ is equivalent to the following four equations:
\begin{equation}\label{8}s_1v_1+s_2v_2=A-c_1-c_2-c_3,\ \ \ s_1+s_2=-B,\end{equation}
\begin{equation}\label{9}v_1c_1-b_1v_1+v_2c_2-b_2v_2+v_1s_1v_1+v_2s_2v_2=C,\ \
v_1s_1+v_2s_2=-D+b_1+b_2+b_3.
\end{equation}
Set $c:=c_1+c_2+c_3$, $b:=b_1+b_2+b_3$ and
$$s:=s_1,\ \ \ v:=v_2,\ \ \ w:=v_2-v_1.$$
Then from the second equation in (\ref{8}) we get $s_2=-(B+s)$; using this, the other
three equations (\ref{8}), (\ref{9}) can be rewritten as
\begin{equation}\label{10}Bv+sw=c-A,\ \ \ \ \ vB+ws=D-b,\end{equation}
\begin{equation}\label{11}v(c_1+c_2-sw)-(b_1+b_2+ws)v-wc_1+b_1w+wsw-vBv=C.\end{equation}
From (\ref{10}) we have that $c_1+c_2-sw=A-c_3+Bv$ and $b_1+b_2+ws=D-b_3-vB$, hence (\ref{11})
can be rewritten as
\begin{equation}\label{12}wsw-wc_1+b_1w=C-v(A-c_3)+(D-b_3)v-vBv.\end{equation}
We are going to show that the system of equations (\ref{10}), (\ref{12}) has a solution.

First suppose that $T$ is not a compact perturbation of a scalar. Then 
we may assume that in the matrix representation of $T$ we have that $D=0$ and that $B$
is an isometry with the range of $B$ isomorphic to its orthogonal complement in $\k$ since
by \cite[Corollary 3.4]{BP} $T$ is similar to such an operator.
In this case we shall see that we can even afford to choose $s=0$, so that the above system 
of equations simplifies to
\begin{equation}\label{13}Bv=c-A,\end{equation}
\begin{equation}\label{14}vB=-b,\end{equation}
\begin{equation}\label{15}b_1w-wc_1=C-v(A-c_3)+(-b_3)v-vBv.\end{equation}
Since $B^*B=1$, the equation (\ref{13}) is equivalent to the following two:
\begin{equation}\label{16}v=B^*(c-A)\ \ \mbox{and}\ \ P^{\perp}(c-A)=0,\ \mbox{where}\ 
P:=BB^*\ \mbox{and}\ P^{\perp}:=1-P.\end{equation}
Using this expression for $v$, (\ref{14}) can be rewritten as
\begin{equation}\label{17}b_1+b_2+b_3=b=B^*(A-c)B.\end{equation}
If there exist $v$, $c_j$ and $b_j$ ($j=1,2,3$) such that the equations (\ref{16}) and
(\ref{17}) are satisfied and the spectra of $c_1$ and $b_1$ are disjoint, then the
equation (\ref{15}) can be solved for $w$ by Rosenblum's theorem.

To show that the system (\ref{16}), (\ref{17}) has a solution, represent $A$ by a
$2\times 2$ operator matrix with respect to the decomposition $\k=P\k\oplus P^{\perp}\k$.
By Lemma \ref{le1} $A=\sum_{j=1}^4A_j$ where each $A_j$ is similar to a positive operator; 
moreover, by Remark \ref{re1} we may assume that  (with respect
to the decomposition $\k=P\k\oplus P^{\perp}\k$) $A_4$ is of the form
\begin{equation}\label{100}A_4=\left[\begin{array}{cc}
a&r\\
0&0\end{array}\right],\ \mbox{where}\ \ a\geq0,\end{equation}
which means that $P^{\perp}A_4=0$.
Thus, if we put $c_j=A_j$ for $j=1,2,3$ (and $c=c_1+c_2+c_3$), then we have $P^{\perp}(A-c)
=P^{\perp}A_4=0$, which is just the condition in (\ref{16}). Further
\begin{equation}\label{int}B^*(A-c)B=B^*A_4B=B^*A_4PB=B^*GB,\end{equation}
where 
$$G:=A_4P=a\oplus0.$$
Thus the operator $B^*(A-c)B$ is positive and hence it can be written (in many ways) as
a sum of three positive operators $b_j$, which is just what the condition (\ref{17}) requires. We may
choose $b_3=0$.
To see that it is possible to choose $b_j$ and $c_j$ ($j=1,2$) so that their spectra
are disjoint, note
that $PB$ is a unitary operator from $\k$ onto $P\k$ which intertwines $a$ and
$b=A-c$ by (\ref{int}), hence $b$ and $a$ have the same spectrum. By 
Lemma \ref{le1} we may choose $a$ and $c_j=A_j$ so that each of their spectra consists of
at most two points, $\sigma(a)\subseteq (0,1]$ and 
$\sigma(A_j)\subset(1,\infty)$ ($j=1,2,3$). Since $b_j\geq 0$ 
and $b_1+b_2=b$, the spectra of
$b_j$ are contained in $[0,1]$, hence $\sigma(b_j)\cap\sigma(c_j)=\emptyset$. Since $\sigma(b)$
consists of at most two points in $(0,1]$, we may choose $b_1, b_2$ to have the same
property. (We may choose for $b_1$ a sufficiently small positive scalar, for example.)

Since $T_j$ is similar to $a_j\oplus b_j$ and $a_j$ is similar to $c_j=A_j$ ($j=1,2,3$),
$\sigma(T_j)=\sigma(A_j)\cup\sigma(b_j)$ consists of at most four points. Other properties of
operators $T_j$ stated in the theorem also follows easily from that of $c_j$ and $a_j$ chosen
above.

Now we consider the case when $T$ is a compact perturbation of a scalar. In this case
let $E=1\oplus0$, 
the projection onto the first summand in the decomposition $\h=\k\oplus\k$.
Then $\tilde{T}:=T-E$ is not a compact perturbation of a scalar, so by the already proved
case $\tilde{T}$ can be expressed as
$\tilde{T}=\sum_{j=1}^3S_j(a_j\oplus b_j)S_j^{-1},$ where $a_j\geq0$ and $b_j\geq0$
and $S_3$ is block-diagonal. Since $S_3$ commutes with $E$, we have
$$T=\tilde{T}+E=\sum_{j=1}^2S_j(a_j\oplus b_j)S_j^{-1}+S_3((a_3\oplus b_3)+E)S_3^{-1},$$
which is a sum of three operators similar to positive ones with (at most) four-point spectra.
\end{proof}

\begin{re}\label{re20}Observe that in the proof of Theorem \ref{th2} the operator $T_3$
is of the form $e\oplus0$, where $e$ is similar to a positive invertible operator with at most two-point
spectrum.
\end{re}

\begin{co}\label{co4}Each $T\in\bh$ can be expressed as $T=\sum_{j=1}^3A_jB_j$,
where $A_j,B_j\in\bh^+$.
\end{co}

\begin{theorem}\label{th3}If $T\in\bh$ is not a compact
perturbation of a scalar, then $T$ is a sum of two operators similar to positive operators.
\end{theorem}

\begin{proof}We have to show that in the  proof of Theorem \ref{th2} $a_3$ and $b_3$ can
be taken to be $0$. That $b_3$ can be taken to be $0$ has been already observed in that proof. 
Now note that in the matrix representation (\ref{3}) of $T$ we may
assume, in addition to $D=0$ and $B$ is an isometry, that $A$ is not a compact perturbation 
of a scalar. For this, we simply decompose the second copy of $\k$ into two orthogonal isomorphic
closed subspaces, $\k=\k_0\oplus\k_1$, and decompose $\h$ as $\h=\k_1^{\perp}\oplus\k_1$. Since
$B$ maps $\k_1$ isometrically into $\k_1^{\perp}$ the matrix of $T$ has $0$ on the $(2,2)$
position and an isometry with infinitely codimensional range on the $(1,2)$ position. 
The new element on the position $(1,1)$ is than
not a compact perturbation of a scalar. So we will assume that already in the initial
matrix representation of $T$ the element $A$ is not a compact perturbation of a scalar.
Consider now the matrix of $A$ relative to the
decomposition of the Hilbert space of $A$ into the range of  $B$ and its orthogonal
complement. Since $A$ is not a compact perturbation of a scalar, by Theorem \ref{th2} and 
Remark \ref{re20} $A$ is of the form $A=\sum_{j=1}^3\tilde{A}_j$, where $\tilde{A}_1$ and
$\tilde{A}_2$ are similar to positive invertible operators each with at most four points in its
spectrum and $\tilde{A}_3$ is of the form $e\oplus0$ with $e$ similar to a positive invertible
operator with a two-point spectrum. By 
the same reasoning as in the proof of Theorem \ref{th2} (see the paragraph containing  (\ref{100});
the role of
$A_4$ is now played by $\tilde{A_3}$) we see that the system of
equations (\ref{16}), (\ref{17}) has a solution such that $c_j=\tilde{A}_j$ for $j=1,2$ and
$c_3=0=b_3=0$. But we have to show also that we can achieve $\sigma(c_j)\cap\sigma(b_j)=
\emptyset$ ($j=1,2$) in order to assure that (\ref{15}) has a solution for $w$ and that
$x_j$ can be computed from the last equation in (\ref{7}). For this we note now  that the operator
$B^*(A-c)B=B^*\tilde{A}_3B$ is unitarily equivalent to  $e$. Since $\sigma(c_j)$ ($j=1,2$)
is a finite subset of $(0,\infty)$ and $\sigma(B^*(A-c)B)$  consists of just two 
positive points, it follows that $B^*(A-c)B$ can be expressed as a sum
$b_1+b_2$, where $b_j\geq0$ and
$\sigma(b_j)\cap\sigma(c_j)=\emptyset$ for both $j=1,2$. 
\end{proof}

An operator $T\in\bh$ of the form $\lambda+K$, where $\lambda\in\bc\setminus\br^+$ and
$K$ is compact, is not of the form 
\begin{equation}\label{P}PQ+RS\ \ \ \mbox{for any}\ P,Q,R,S\in\bh^+.\end{equation} To see 
this, just note
that the spectrum of  the coset $\dot{R}\dot{S}$ in the Calkin algebra is the same as
the spectrum of $\dot{S}^{1/2}\dot{R}\dot{S}^{1/2}$, hence contained in $\br^+$, while
the spectrum of $\lambda-\dot{P}\dot{Q}$ is contained in the ray $\lambda-\br^+$
which is disjoint with  $\br^+$. 

Each compact operator on a Hilbert space is an additive commutator of two bounded operators
\cite{AS}. By an analogy one might conjecture that each compact operator is a sum of
two operators similar to positive ones, but this is not true.

\begin{pr}If $T\in{\rm C}^1(\h)$ (the trace class) is nonzero and $\tr(T)$ is not positive, then $T$ is 
not a sum of two operators in $\bh$ similar to positive ones.
\end{pr}

\begin{proof}Assume the contrary, that $T=S_1AS_1^{-1}+S_2BS_2^{-1}$, where $A,B\in\bh^+$.
Put $F:=-S_1^{-1}TS_1$  and $S=S_1^{-1}S_2$. Then
\begin{equation}\label{50}F+A=-SBS^{-1}.\end{equation}
Considering the essential spectra, it follows from (\ref{50}) and the positivity of $A$
and $B$ that $A$ and $B$ must be compact. We claim, that $A$ and $B$ must be in the
Hilbert-Schmidt class ${\rm C}^2(\h)$. For a proof we may first replace $B$ by a unitarily
equivalent operator (and modify $S$ accordingly) to reduce to the situation when $A$
and $B$ can be diagonalized in the same orthonormal basis $\mathbb{B}$ of $\h$. Let $(\alpha_j)$ and
$(\beta_j)$ be the lists of eigenvalues of $A$ and $B$ in decreasing order (each eigenvalue
repeated according to its multiplicity). From (\ref{50}) we have $AS+SB=G$, where $G:=-FS$. 
Denoting by $\sigma_{i,j}$ and $\psi_{i,j}$ the entries of the matrices of $S$ and $G$ in the basis 
$\mathbb{B}$, this means that
\begin{equation}\label{51}(\alpha_i+\beta_j)\sigma_{i,j}=\psi_{i,j}.\end{equation}
Let $\gamma_j:=(\sum_i|\psi_{i,j}|^2)^{1/2}$ and note that $\sum_j\gamma_j^2<\infty$ since $G\in{\rm C}^2(\h)$.
Since $S$ is invertible (in particular, bounded from below), there exists a scalar $\gamma>0$ such
that $\sum_i|\sigma_{i,j}|^2\geq \gamma$ for all $i$, hence (\ref{51}) implies that
$$\beta_j^{-2}\gamma_j^2=\beta_j^{-2}\sum_i|\psi_{i,j}|^2
=\sum_i\frac{(\alpha_i+\beta_j)^2}{\beta_j^2}|\sigma_{i,j}|^2
\geq\sum_i|\sigma_{i,j}|^2\geq \gamma,$$
whenever $\beta_j\ne0$. Thus $\beta_j^2\leq|\gamma_j|^2\gamma^{-1}$ and consequently 
$\sum_j \beta_j^2<\infty$,
which means that $B\in{\rm C}^2(\h)$. Similarly (or from (\ref{50}), since $F\in{\rm C}^2(\h)$)
we see that $A\in{\rm C}^2(\h)$.

By considering the polar decomposition of $S$ of the form $S=RU$, where $R$ is positive and
$U$ is unitary, we may rewrite (\ref{50}) in the form
\begin{equation}\label{52}F+A=-RCR^{-1},\end{equation}
where $C:=UBU^*\geq0$. Assume for a moment that in some orthonormal basis of $\h$ the operator 
$R$ can be represented by a diagonal matrix
and let $[\alpha_{i,j}]$, $[\phi_{i,j}]$ and
$[\gamma_{i,j}]$ be the matrices of $A$, $F$ and $C$ (respectively) in this basis. Then, considering
the sums of diagonal terms of matrices, (\ref{52}) implies that
\begin{equation}\label{53}\sum_{j=1}^n\psi_{j,j}+\sum_{j=1}^n\alpha_{j,j}=-\sum_{j=1}^n\gamma_{j,j}.
\end{equation}
Letting $n\to\infty$, the first sum in (\ref{53}) tends to $\tr(F)=\tr(T)\in\bc\setminus(0,\infty)$,
while the second and the third sums converge to elements in $[0,\infty]$. This shows that
the equality (\ref{53}) can hold for all $n$ only if $\tr(T)=0$ and $\psi_{j,j}=0=\alpha_{j,j}$
for all $j$. Since $A\in\bh^+$, the condition $\alpha_{j,j}=0$ for all $j$ implies that $A=0$.
But then $B$ is similar to $T$, hence $\tr(B)=0$, which implies (since $B\geq0$) that $B=0$. 
In this case $T=0$, which was excluded by the hypothesis of the proposition. Now we
will show by an approximation argument that (\ref{52}) leads to a contradiction even if $R$ can not be diagonalized.

By the Weyl - von Neumann theorem \cite[p. 214]{Co2}, given $\varepsilon>0$, there exist a diagonal
hermitian operator $D$ and an operator $H\in{\rm C}^2(\h)$ with $\|H\|_2<\varepsilon$ 
(where $\|\cdot\|_2$ denotes the Hilbert - Schmidt norm) such that $R=D+H$. If $\varepsilon$ is
small enough then $D$ is invertible (since $D=R-H=R(1-R^{-1}H)$) and
$$\|D^{-1}\|\leq\|R^{-1}\|\sum_{n=0}^{\infty}\|R^{-1}H\|^n\leq\frac{\|R^{-1}\|}{1-\varepsilon\|R^{-1}\|}.$$
Further, if $\varepsilon$ is small enough then $1+HD^{-1}$ is invertible and
$$RCR^{-1}=(1+HD^{-1})DCD^{-1}(1+HD^{-1})^{-1}.$$
Since $(1+HD^{-1})^{-1}=1-HD^{-1}(1+HD^{-1})^{-1}$, we may write
\begin{multline*}RCR^{-1}
=DCD^{-1}-DCD^{-1}HD^{-1}(1+HD^{-1})^{-1}\\+HCD^{-1}\left[1-HD^{-1}(1+HD^{-1})^{-1}\right],
\end{multline*}
hence (since $B$ and therefore also $C$ is in ${\rm C}^2(\h)$ by the first paragraph of this proof)
\begin{multline*}\|RCR^{-1}-DCD^{-1}\|_1 \leq\|H\|_2\|C\|_2\|D^{-1}\|\cdot\\
\left[\|D\|\|D^{-1}\|\|(1+HD^{-1})^{-1}\|+
\|\|1-HD^{-1}(1+HD^{-1})^{-1}\|\right].\end{multline*}
It follows that $\|RCR^{-1}-DCD^{-1}\|_1\to0$ as $\varepsilon\to0$. This allows us to conclude
in essentially the same way as in the previous paragraph (by considering the sums of diagonal entries of
matrices) that (\ref{52}) leads to a contradiction.
\end{proof}

For most of the above proof it would be sufficient if we assumed that $T\in{\rm C}^2(\h)$
(instead of $T\in{\rm C}^1(\h)$), but the problem is that for an operator $T$ not in ${\rm C}^1(\h)$
the sum of diagonal entries of its matrix relative to a general orthogonal basis can be quite
arbitrary (it need not even be defined \cite{FH}).

\medskip
{\bf Problem.} Which compact operators on an infinite dimensional Hilbert space
can be written as $T_1+T_2$, where $T_1$ and $T_2$ are similar to positive operators? 

\medskip
Theorem \ref{th3} implies that all operators can be approximated in norm by sums
of two operators similar to positive ones; but concerning such approximation a much stronger
result holds: it follows from \cite[Theorem 3.10]{DM} that both summands can be 
taken to be similar to the same positive operator.

\section{On spectra of L\" uders operators}

For two commutative $m$-tuples $(A_j)$ and $(B_j)$ of elements of $\bh$ the
spectrum $\sigma(\Phi)$ of the map $\Phi(X):=\sum_{j=1}^mA_jXB_j$ on $\bh$ 
can be described in terms of spectra of $(A_j)$ and $(B_j)$ (\cite{Cur}, \cite{N}); 
in particular $\sigma(\Phi)\subseteq\br^+$ if $A_j,B_j\in\bh^+$.
For noncommutative $(A_j)$ and $(B_j)$ the situation may be completely different. One consequence of Theorem \ref{th2} is that for an infinite dimensional Hilbert space
$\h$ the spectra of L\" uders operators on $\bh$ are not necessarily contained in $\br^+$.

\begin{pr}\label{pr5}Let $\h$ be an infinite dimensional Hilbert space. Every complex 
number $\lambda$ can be an eigenvalue of a L\" uders operator on $\bh$ of length $3$ (or more).
\end{pr}

\begin{proof}Decompose $\h$ as $\h=\k\oplus\k$. By Corollary \ref{co4} there exist
$A_j,B_j\in\bk^+$ such that $\sum_{j=1}^3A_jB_j=\lambda$. By a simple calculation this 
implies that the operator
$$X_0:=\left[\begin{array}{cc}
0&1\\
0&0\end{array}\right]$$
is an eigenvector corresponding to the eigenvalue $\lambda$ of the L\" uders operator 
$\Phi$ on $\bh$ defined by $\Phi(X)=\sum_{j=1}^3T_jXT_j$, where
$$T_j=\left[\begin{array}{cc}
A_j&0\\
0&B_j\end{array}\right].$$
\end{proof}

\begin{theorem}\label{th6}Suppose that $A_j,B_j\in\bh^+$ ($j=1,2$) and let $\Phi$ be the
map on $\bh$ defined by $\Phi(X)=\sum_{j=1}^2A_jXB_j.$ If $A_1A_2=A_2A_1$ (or if
$B_1B_2=B_2B_1$) then the spectrum of $\Phi$ is contained in $\br^+$.
\end{theorem}

\begin{proof}Since boundary points of the spectrum of any operator 
are approximate eigenvalues \cite{Co}, it suffices to show that each approximate eigenvalue
$\lambda$ of $\Phi$ is in $\br^+$. By considering the space $\mathcal{B}:=\ell_{\infty}(\bh)/c_0(\bh)$,
where $\ell_{\infty}(\bh)$ is the space of all bounded sequences with the entries in $\bh$
and $c_0(\bh)$ is the subspace of all sequences converging (in norm) to $0$, we may reduce
the approximate eigenvalues of $\Phi$ to proper eigenvalues of the corresponding operator
$\tilde{\Phi}$ on $\mathcal{B}$. Here of course $\tilde{\Phi}$ is defined by $\tilde{\Phi}([X_n])=
[\Phi(X_n)]$, where $[X_n]$ denotes the coset of a sequence $(X_n)\in\ell_{\infty}(\bh)$.
Note that $\tilde{\Phi}$ is again an elementary operator, namely of the form
\begin{equation}\label{31}\tilde{\Phi}(Y)=\sum_{j=1}^2\tilde{A}_jY\tilde{B}_j\ \ \ (Y\in B),
\end{equation} where $\tilde{A}$ denotes the coset in $\mathcal{B}$ of the constant sequence 
$(A,A,\ldots)\in\ell_{\infty}(\bh)$ for each $A\in\bh$.
Since $\mathcal{B}$ is a $C^*$-algebra, we can regard it as a subalgebra of $\bk$ for some 
(non-separable) Hilbert space $\k$ and by the formula (\ref{31}) we may regard the map 
$\tilde{\Phi}$ to be defined on all $\bk$. Any approximate eigenvalue $\lambda$ of $\Phi$ is
then an eigenvalue of $\tilde{\Phi}$. Choose a nonzero eigenvector $Y$
corresponding to $\lambda$. $\k$ is not 
separable, but it can be expressed as an orthogonal sum of separable subspaces $\k_i$
that reduce all the operators $A_j$, $B_j$ and $Y$. If $i$ is such that $Y|\k_i\ne0$,
then $\lambda$ is an eigenvalue of the operator $\Psi$ on ${\rm B}(\k_i)$ defined by
$\Psi(X)=\sum_{j=1}^2C_jXD_j$, where $C_j=A_j|\k_i$ and $D_j=B_j|\k_i$. So it suffices
to show that all eigenvalues of such operators are in $\br^+$. Thus, (adapting the notation)
we may assume that $\lambda$ is an eigenvalue of $\Phi$. Denote by $X$ a corresponding
eigenvector with $\|X\|=1$, hence
\begin{equation}\label{32}\sum_{j=1}^2A_jXB_j=\lambda X.\end{equation}

Suppose that $A_1$ and $A_2$ commute. Then by Voiculescu's version \cite{Voi}  of the 
Weyl-von Neumann-Berg theorem, given $\varepsilon
>0$, there exist commuting diagonal hermitian operators $C_j\in\bh$ and Hilbert-Schmidt 
operators  $H_j\in {\rm C^2(\h)}$ such that $A_j=C_j+H_j$ and $\|H_j\|_2<\varepsilon$
($j=1,2$). Let $C_j=C_j^+-C_j^-$ be the decomposition of $C$ into the positive and the negative part and
denote by $Q_j$ the range projection of $C_j^-$. Then $A_j+C_j^-=C_j^++H_j$, hence
(since $Q_jC_j^+=0$ and $Q_jC_j^-=C_j$)
$$Q_jA_jQ_j+C_j^-=Q_jH_jQ_j\in{\rm C^2(\h)}.$$
This implies that $C_j^-\in{\rm C^2(\h)}$ and $\|C_j^-\|_2\leq\|H_j\|_2<\varepsilon$.
So, replacing $C_j$ by $C_j^+$ and $H_j$ by $H_j-C_j^-$ (and the initial $\varepsilon$
by $\varepsilon/2$), we may assume that $C_j\geq0$. Let $P$ be any finite rank projection
that commutes with $C_1$ and $C_2$. (Note that, since $C_1$, $C_2$ are commuting diagonal
operators, there exist a net of such projections $P$ converging strongly to the identity.)
From (\ref{32}) we have that $\sum PA_jXB_jX^*P=\lambda PXX^*P$, hence applying the trace
$\tr$ we obtain
\begin{equation}\label{33}\sum_{j=1}^2\left(\tr(PC_jXB_jX^*P)+\tr(PH_jXB_jX^*P)\right)=
\lambda\tr(PXX^*P).\end{equation}
Since $P$ commutes with $C_j$, 
\begin{equation}\label{40}tr(PC_jXB_jX^*P)=\tr(C_jPXB_jX^*P)
=\tr(C_j^{1/2}PXB_jX^*PC_j^{1/2})\geq0.\end{equation}
Further (since $\|Z\|_2=\|Z^*\|_2$ for all $Z\in\bh$) ,
\begin{equation}\label{34}|\tr(PH_jXB_jX^*P|\leq\|H_j\|_2\|XB_jX^*P\|_2=\|H_j\|_2
\|PXB_jX^*\|_2<\varepsilon\|PX\|_2,\end{equation}
where we have assumed (without lost of generality) that $\|B_j\|\leq1$.
If $P$ is sufficiently close to $1$ so that $PX\ne0$, then from (\ref{33}) and
(\ref{34}) we have that
$$\begin{array}{lll}
\left|\lambda-\sum_{j=1}^2\frac{\tr(PC_jXB_jX^*P)}{\tr(PXX^*P)}\right|&\leq&\varepsilon
\sum_{j=1}^2\frac{\|PX\|_2}{\tr(PXX^*P)}\\
 &=&\frac{2\varepsilon}{\|PX\|_2}.\end{array}$$
Letting  in this estimate $P\to 1$, $\varepsilon\to0$ and using (\ref{40}), we see
that $\lambda\geq0$.
\end{proof}

\begin{re}Theorem \ref{th6} can be extended to operators of the form \begin{equation}\label{99}
X\mapsto
\sum_{j=1}^nA_jXB_j\end{equation} if the coefficients on one side, say all the $A_j$, are smooth
nonnegative functions $A_j=f_j(H_1,H_2)$ of a pair of commuting hermitian operators $(H_1,H_2)$. 
Namely, in this case it can be shown (using the Fourier transform) that
small Hilbert-Schmidt perturbations of $(H_1,H_2)$ result in small Hilbert-Schmidt perturbations
of $f_j(H_1,H_2)$. The author does not know if the theorem can be extended to
the general situation, when all the $A_j$ commute, but the $B_j$ do not necessarily commute.
\end{re}

{\bf Problems.} 1. Can Theorem \ref{th6} be generalized to operators of length greater than $2$?

\medskip 2. Suppose that all $A_j,B_j$ are positive and for each $j$ at least one of
$A_j, B_j$ is compact. Then it can be deduced from \cite[Corollary 6.6]{Sh} 
(see \cite{M}) that all  eigenvalues of the operator 
(\ref{99}) are contained in $\br^+$. Is the same true for the entire spectrum?

3. Can in Theorem \ref{th6} the commutativity condition be replaced by commutativity
modulo compact operators?

\end{document}